\documentclass{article}

\usepackage{graphicx} 
\usepackage{amsmath}
\usepackage{amsthm}
\usepackage{amsfonts}
\usepackage{amssymb}
\usepackage{mathtools}
\newtheorem{theorem}{Theorem}
\newtheorem{lemma}[theorem]{Lemma}
\newtheorem{proposition}[theorem]{Proposition}
\newtheorem{corollary}[theorem]{Corollary}
\newtheorem{definition}{Definition}

\usepackage{url}
\usepackage{comment}

\usepackage{comment}
\usepackage{xfrac}
\usepackage{algorithm}
\usepackage{algpseudocode}
\usepackage{xcolor}
\usepackage{tikz-cd} 
\usepackage{booktabs}   
\usepackage{graphicx}   
\usepackage{float}      
\usepackage{xfrac}   
\usepackage[numbers]{natbib}
\usepackage{booktabs}
\usepackage{tabularx}
\usepackage{array}

\newcolumntype{L}{>{\raggedright\arraybackslash}X}
\newlength{\tblwidth}
\allowdisplaybreaks

\title{Wedderburn decomposition of Twisted skew abelian group algebra}
\author{Alvaro Otero Sanchez, aos073@ual.es}

\begin{document}
\maketitle

\begin{abstract}
We compute the Wedderburn decomposition of a finite twisted skew group algebra over a finite abelian group $G$ and a finite field $K$. To do so, we prove that this twisted skew group algebra is a homomorphic image of a finite group algebra. We then characterise its Shoda pairs, as well as propose an algorithm for computing a complete set of them.
\end{abstract}


\section{Introduction}
In \cite{Wedderburn1908}, Wedderburn proved that all simple finite-dimensional algebras over a field are isomorphic to the matrix algebra over a division ring. In \cite{Artin1927}, Artin extended that result by proving that all finite-dimensional semisimple algebras are isomorphic to a direct sum of matrix algebras over possibly different division rings. Since then, this direct sum is called the Wedderburn decomposition of the algebra, and there has been a vast literature on how to compute them \cite{Pierce1982}, \cite{Behboodi2018}, \cite{Bade1992}, \cite{DeGraaf1997}.

One algebra that has a major impact on mathematics is the concept of the group ring. While its origins can be traced back to \cite{Molien1897}, the study of its Wedderburn decomposition started with the foundational result of Maschke \cite{Maschke1899}, where the author proved several results about the representation of finite groups over a field of characteristic $0$. This result was the origin of Maschke's theorem, which establishes the necessary and sufficient conditions for a group algebra to be semisimple. Since then, group rings have been closely connected to representation theory, providing an algebraic tool to understand the representation of groups by matrices over a ring. 

Group rings find a notable application in coding theory through the development of 
so-called group ring codes. Hurley \cite{HurleyHurley2009} initially proposed 
a method for constructing linear codes from group rings; however, a unified theory 
alongside their fundamental properties can be found in \cite{Hughes2001}. 
Recently, these concepts have been extended to twisted group rings. For instance, 
\cite{BehajainaBorelloDeLaCruz2024} demonstrated that a wide variety of classical codes can be understood as specific instances of twisted skew group ring codes.

One of the central problems in the study of group algebras and their generalizations is obtaining their Wedderburn decomposition. A standard method is to compute the characters of the group \cite{Yamada1974}; however, this process requires knowledge of the character table of the group under consideration. In \cite{JespersLealPaques2003}, the authors proposed a character-free method to compute the idempotents of the group algebra of a finite group over the rational numbers. In \cite{BROCHE200771}, this result was extended to finite fields, which led to computational implementations such as Wedderga, a GAP package designed to compute these idempotents \cite{Wedderga}. Nevertheless, there is still a lack of research on the case of twisted skew group rings.

Recently, there have been some major advances in the twisted case, such as the computation of the Wedderburn decomposition of twisted abelian group algebras \cite{DUARTE2024102386} or twisted dihedral group algebras \cite{DUARTE2026102792}. However, these works take a ring-theoretic approach and are limited to the twisted case. In this work, we will focus on a more group-theoretic approach, proving how we can compute the Shoda pairs of a certain type of metabelian group and use them for twisted skew group algebras over a finite field.

\section{Preliminars}
In this section, we introduce the relevant notation and review key properties of groups that will be used in this article. In addition, we will recall results on the Wedderburn decomposition of group rings
\begin{definition}
    Let $G$ be a finite group and $R$ be a commutative ring. The group algebra $RG$ is the free $R-$module generated by the elements of $G$ and multiplication given by
    \begin{equation}
        (a \overline{g})(b \overline{h}) = ab \overline{gh} \quad \forall a,b \in R, g,h \in G
    \end{equation}
    and extended linearly. 
\end{definition}

The idea of the group ring can be generalized to the concept of twisted skew group ring \cite{Karpilovsky1989}. To do so, we need the following definition

\begin{definition}
    Let $G$ be a group and $F$ be a field. Suppose that there exist an action $\sigma:G \longrightarrow Aut(F)$ that will be denoted by $b^g = \sigma(g)(b)$. A map $\alpha:G\times G \longrightarrow F^*$ is called a $2-$cocycle for the action $\sigma$ if 
    \begin{equation}
        \alpha(x,y)\alpha(xy,z)= \alpha(y,z)^x\alpha(x,yz) \quad \forall x,y,z\in G
    \end{equation}
    We will say that the a $2-$cocycle $\alpha$ is normaliced if $\alpha(1_G,g) = \alpha(g,1_G) = 1_F$. The set of all $2-$cocycle will be denoted by $Z^2(G,F)$.
\end{definition}

Now we can define a twisted skew group ring

\begin{definition}
Let $G$ be a group, $F$ be a field, $\sigma:G \longrightarrow Aut(F)$ an action of $G$ over $F$ and let $\alpha\in Z^2(G,F)$. The twisted skew group algebra of $F$ over $G$, denoted by $F^\alpha [ G; \sigma]$, is the free $F-$module generated by the elements $g\in G$, denoted by $\overline{g}$, with the inner operation

\begin{equation}
(a \overline{x}) \cdot (b \overline{y}) = a \prescript{x}{}{b} \alpha(x,y) \overline{xy} \forall a,b \in F, x,y \in G
\end{equation}

\noindent and extended by linearity.
\end{definition}
Note that the twisted skew group algebra is also called crossed product in the literature \cite{Karpilovsky1989}.

It is a well know fact that this structure is associative \cite{Passman1989}.
\begin{proposition} 
    The twisted skew group ring is an associative ring.
\end{proposition}

To understand the structure of twisted skew group rings, we need to introduce the cohomology of a group \cite{brown1982cohomology}.
\begin{definition}
    Let $G$ be a group, $F$ be a field and $\sigma:G \longrightarrow Aut(F)$ a group action that will be denoted by $b^g = \sigma(g)(b)$. A $2-$cocycle $\alpha\in Z^2(G,F)$ is said to be a $2-$coboundary if there exist $t:G\longrightarrow F^*$ such that
    \begin{equation*}
        \alpha(g,h) = t(g)t(h)^g t(gh)^{-1} \quad \forall g,h\in G
    \end{equation*}
    The set of all $2-$coboundaries will be denoted by $B^2(G,F)$
\end{definition}

The following result proves that coboundaries play a crucial role in the structure of twisted group algebras.
\begin{proposition}[\cite{linckelmann2018block}, Proposition 1.2.6]
     Let $G$ be a group,  $F$ be a field  and $\sigma:G \longrightarrow Aut(F)$ a group action. Then
     \begin{enumerate}
     \renewcommand{\labelenumi}{(\roman{enumi})}
         \item If $\alpha,\beta\in Z^2(G,F)$, then the product given by 
         \begin{align*}
             \alpha\beta: G\times G &\longrightarrow F \\
             (g,h) &\longmapsto \alpha(g,h) \beta(g,h)
         \end{align*}
     
     gives to $Z^2(G,F)$ the structure of abelian group.
     \item $B^2(G,F)$ is a normal subgroup of $Z^2(G,F)$, and we will call it quotient $H^2(G,F)=Z^2(G,F)/B^2(G,F)$ the second cohomology group of $G$ with coefficient in $F$.
     \item Let  $\alpha,\beta\in Z^2(G,F)$. Then $F^\alpha [G,\sigma] \simeq F^\beta [G,\sigma]$ if and only if $[\alpha] = [\beta]$ in $H^2(G,F)$.
     \end{enumerate}
\end{proposition}

In addition, it is well known that every class in $H^2(G,F)$ contains a normalized representative. Thus, up to isomorphism, we may assume without loss of generality that all $2$-cocycles in this work are normalized. For further details on the cohomology of groups and its computation, see \cite{brown1982cohomology}.

From \cite{JespersLealPaques2003} we can get the following definition

\begin{definition}
Let $G$ be a finite group of exponent $n$, and let $q$ be an integer coprime to $n$. Consider the subgroup $Q = \langle q + n\mathbb{Z} \rangle \subseteq \mathbb{Z}_n^\times$. The group $Q$ acts on $G$ via power maps, given by
\[
  (m + n\mathbb{Z}) \cdot g = g^m \quad \text{for all } m + n\mathbb{Z} \in Q, \, g \in G.
\]
The \emph{$q$-cyclotomic class} of an element $g \in G$, denoted by $C_q(g)$, is the orbit of $g$ under this action, that is,
\[
  C_q(g) = \left\{ g^{q^k} : k \ge 0 \right\} = \left\{ g, g^q, g^{q^2}, \dots, g^{q^{o-1}} \right\},
\]
where $o = \operatorname{ord}_d(q)$ is the multiplicative order of $q$ modulo $d = |g|$.
\end{definition}

Let $q$ be a prime power and let $k$ be a positive integer coprime to $q$. We denote by $\xi_k$ a primitive $k$-th root of unity over $\mathbb{F}_q$. We let $o_k = o_k(q) = \operatorname{ord}_k(q)$ denote the multiplicative order of $q$ modulo $k$.

In particular, the field extension generated by $\xi_k$ over $\mathbb{F}_q$ is given by
\[
  \mathbb{F}_q(\xi_k) \cong \mathbb{F}_{q^{o_k}}.
\]

\begin{definition}
Let $N \trianglelefteq G$ be such that $G/N$ is cyclic of order $k$, and let
$C \in \mathcal{C}(G/N)$. If $\xi \in C$ and
\[
\operatorname{tr}=\operatorname{tr}_{F(\zeta_k)/F}
\]
denotes the trace of the field extension $F(\zeta_k)/F$, then we set
\[
\varepsilon_C(G,N)
=
\frac{1}{|G|}
\sum_{g\in G}
\operatorname{tr}(\xi(g))\,g^{-1}
=
\frac{1}{[G:N]}
\widehat{N}
\sum_{X\in G/N}
\operatorname{tr}(\xi(X))\,g_X^{-1},
\]
where $\overline{g}$ denotes the image of $g$ in $G/N$, and $g_X$ denotes a
representative of $X\in G/N$.

Let $H \trianglelefteq K \leq G$ such that $K/H$ is cyclic, and let
$C \in \mathcal{C}(K/H)$. Then $e_C(G,K,H)$ denotes the sum of the distinct
$G$-conjugates of $\varepsilon_C(K,H)$. In addition, if we denote $N = N_G(H) \cap N_G(K)$, then $N$ acts on $K/H$ by conjugation and this induces an action of $N$ on the set of $q$-cyclotomic classes of $K/H$. We denote by $E_G(K/H)$ the 
stabilizer of any $q$-cyclotomic class of $K/H$ containing generators 
of $K/H$ under this action of $N$.
\end{definition}

Recall that a group $G$ is \emph{metabelian} if its derived subgroup $G' = [G,G]$ is abelian or, equivalently, if there exists an abelian normal subgroup $A \triangleleft G$ such that the quotient $G/A$ is abelian. In what follows, we restrict our attention to metabelian groups, as in \cite{BROCHE200771} it is proved the following result. 
\begin{proposition} \label{metabelian}
    
Let $G$ be a finite metabelian group, $A$ a maximal abelian subgroup of $G$ containing $G'$, and $F$ a finite field such that $FG$ is semisimple. Then every primitive central idempotent of $FG$ is of the form $e_C(G,K,H)$, where $(K,H)$ is a pair of subgroups of $G$ satisfying the following conditions:

\begin{enumerate}
\item $K$ is a maximal element in the set
\[
\{\, B \leq G \mid A \leq  B  \text{ and}\ B' \leq H \leq B \,\},
\]

\item $K/H$ is cyclic,
\end{enumerate}

and $C \in \mathcal{C}(K/H)$.

Furthermore, for every pair $(K,H)$ of subgroups of $G$ satisfying (1) and (2), and every $C \in \mathcal{C}(K/H)$, we have
\[
FG\, e_C(G,K,H) \cong M_{[G:K]}\big( \mathbb{F}_{q^{\,o/[E:K]}} \big),
\]
where $E = E_G(K/H)$ and $o$ is the multiplicative order of $q$ modulo $[K:H]$.
\end{proposition}

Therefore, in order to compute the Wedderburn decomposition of a metabelian group, it is enough to compute the subgroups that satisfies the conditions of proposition \ref{metabelian}. To relate this problem to the twisted skew abelian group algebra, we have the following result.

\section{Twisted skew group ring as quotient}

Let $G$ be a group, let $p$ be a prime number, and $\mathbb{F}_{p^n}$ be the finite field of $p^n$ elements. Let $\mathbb{F}_{p^n}^\alpha[G,\sigma]$ be a twisted skew group ring. We will prove that it is the quotient of a group ring $\mathbb{F}_p \hat{G}$ for some finite group $\hat{G}$. 

Recall that $\mathrm{Aut}(\mathbb{F}_{p^n}) = \mathrm{Gal}(\mathbb{F}_{p^n}/\mathbb{F}_p) \cong C_n$, the cyclic group of order $n$. This Galois group is generated by the Frobenius automorphism $\sigma \colon \mathbb{F}_{p^n} \longrightarrow \mathbb{F}_{p^n}$, defined by $\sigma(x) = x^p$ for all $x \in \mathbb{F}_{p^n}$. In particular, we have that $\mathbb{F}_p$ is fixed by the action of $G$ and $\mathbb{F}_{p^n}^\alpha[G,\sigma]$ is an algebra over its prime field.

Let $\hat{G}= \mathbb{F}_{p^n}^* \times G$ and let 
\begin{align*}
    \hat{G} \times \hat{G} & \longrightarrow \hat{G} \\
    ((a,g),(b,h)) & \longmapsto (a \prescript{g}{}{b} \alpha(g,h) , gh)
\end{align*}

\begin{proposition} \label{hatGisgroup}
    $\hat{G}$ with the previous operation is a group. 
\end{proposition}
\begin{proof}
    As $G$ and $\mathbb{F}_{p^n}^*$ are groups, we have that $\hat{G}$ is closed under the operation. Also, the associativity comes from the associativity of the product in a twisted skew group ring. 

    It has an identity element, given by $1_{\hat{G}} = (1_{\mathbb{F}_p^n}, 1_G)$ as for all $(x,y) \in \hat{G}$
    \begin{align*}
        (x,y) 1_{\hat{G}} = &  (x,y) (1_{\mathbb{F}_p^n}, 1_G) \\ = & (x \prescript{y}{}{1_{\mathbb{F}_p^n}}, y 1_G )  \\ = & (x,y)  \\ = & (  1_{\mathbb{F}_p^n}\prescript{1_G}{}{x},  1_G y)  \\ = & (1_{\mathbb{F}_p^n}, 1_G)  (x,y)  \\ = & 1_{\hat{G}} (x,y)
    \end{align*}
    And therefore $(x,y) 1_{\hat{G}} = (x,y) = 1_{\hat{G}}  $ as we wanted to prove. 
    Finally, the inverse element of $(x,y) \in \hat{G}$ is given by $\left( \left(\frac{1}{x\alpha(y,y^{-1})}\right)^{y^{-1}},y^{-1} \right)$, as 
    \begin{equation}
        (x,y)  \left( \left(\frac{1}{x\alpha(y,y^{-1})}\right)^{y^{-1}},y^{-1} \right) = \left( \frac{x\alpha(y,y^{-1})}{x\alpha(y,y^{-1})} , y y^{-1}\right) = (1_{\mathbb{F}_{p^n}},1_G) = 1_{\hat{G}}
    \end{equation}
    and the other side is proven analogously.
\end{proof}

We know that $(\mathbb{F}_q^{*},\cdot)$ is a cyclic group of order $q-1$, and let
$\lambda \in \mathbb{F}_q^{*}$ be a generator. If we consider the powers of $\lambda$ within the
field then $\lambda$ satisfies a
polynomial $P(x)\in \mathbb{F}_p[x]$ 
Moreover, the finite field $\mathbb{F}_q$ can be realised as the quotient ring
$\mathbb{F}_q \cong \mathbb{F}_p[x]/\langle P(x)\rangle .
$
\begin{proposition}
Let $G$ be a group, let $p$ be a prime number, and $\mathbb{F}_{p^n}$ be the finite field of $p^n$ elements, and let $\mathbb{F}_{p^n}^\alpha[G,\sigma]$ be a twisted skew group ring. Let $\lambda\in \mathbb{F}_q^{*}$ be a generator of $ \mathbb{F}_q^{*}$ with minimal polynomial $P(x) \in \mathbb{F}_p[x]$. 

Then there is an epimorphism of $\mathbb{F}_p$-algebras over  $\psi:\mathbb{F}_p \hat{G} \longrightarrow \mathbb{F}_{p^n}^\alpha[G,\sigma]$ such that $\ker(\psi) = \langle P((\lambda,1)) \rangle$.
\end{proposition}
\begin{proof}
We will prove that the epimorphism $\psi :\mathbb{F}_p \hat{G}  \longrightarrow \mathbb{F}_{p^n}^\alpha[G,\sigma]$ is given by $\psi(a \overline{(b,c)}) = ab \overline{c}$ for all $a\in \mathbb{F}_p$, $(b,c) \in \hat{G}$ and extended by linearity. It respects addition by definition, so we only have to prove that it respects the product, and to do so we can focus on the basis elements
\begin{align*}
        \psi(\overline{(a,g)} \overline{(b,h)} ) =&  \psi(  (a \prescript{g}{}{b} \alpha(g,h) , gh)  \\ =&  a \prescript{g}{}{b} \alpha(g,h) \overline{gh}  \\ =& ( a\overline{g} ) ( b \overline{h} )  \\ =& \psi(\overline{(a,g)}) \psi ( \overline{(b,h)}
\end{align*}

And therefore, it is a $\mathbb{F}_p$-homomorphism of rings. To prove that it is an $\mathbb{F}_p$-epimorphism, note that as an $\mathbb{F}_p$-algebra, it is generated by $\lambda^i\overline{g}$ for $i=0,\cdots, n-1$ and $g\in G$, and clearly $\lambda^i\overline{g} = \psi((\lambda^i,g))$.

To see that the kernel is given by $\langle P((\lambda,1)) \rangle$, first note that $\psi ( P(\lambda,1)) = P(\lambda) \overline{1} = 0$. For the other implication, let $\sum_{i=0,\cdots, n; g\in G}a_{i,g} ( \lambda^i, g) \in \mathbb{F}_p \hat{G} $ such that $\psi( \sum_{i=0,\cdots, n; g\in G}a_{i,g} ( \lambda^i, g) ) = 0$. Then we have that 
\begin{equation*}
    \psi \left(\sum_{i=0,\cdots, n; g\in G}a_{i,g} ( \lambda^i, g) \right) = \sum_{i=0,\cdots, n; g\in G}a_{i,g}  \lambda^i \overline{g} = \sum_{g\in G} \left(\sum_{i=0}^n  a_{i,g} \lambda^i \right) \overline{g} = 0
\end{equation*}
So $\sum_{i=0}^n  a_{i,g} \lambda^i =0$ for all $g\in G$. As $p(x) $ is the minimal polynomial of $\lambda$, we have that $P(x) | \sum_{i=0}^n  a_{i,g} x^i$ and $\sum_{i=0}^n  a_{i,g} x^i = k_g(x)P(x)$. As a result, 
\begin{align*}
    \sum_{i=0,\cdots, n; g\in G}a_{i,g} ( \lambda^i, g) = & \sum_{i=0,\cdots, n; g\in G}a_{i,g} ( \lambda^i, 1)  (1,g)\\ = & \sum_{g\in G} \left( k_g( \lambda^i, 1) p( \lambda^i, 1) \right) (1,g) \in \langle P(\lambda,1)\rangle
\end{align*}
\end{proof}

Finally, we have the following easy result.
\begin{corollary}
    Let $G$ be an abelian group. Then $\hat{G}$ is a metabelian group.
\end{corollary}
\begin{proof}
    Let $\pi_2 : \hat{G} \longrightarrow G$ be the projection on the second component. It is trivial to check that it is an epimorphism of groups and that $\ker (\pi_2) = \mathbb{F}_{p^n} \times \{1\}$. As the $2$-cocycle is normalised, and $1_G$ acts trivially over $\mathbb{F}_{p^n}$, we have that $\mathbb{F}_{p^n} \times \{1\} \simeq \mathbb{F}_{p^n}$, and therefore it is commutative. In addition, as it is the kernel of a morphism, it is a normal subgroup. Therefore, $\ker(\pi_2) \trianglelefteq \hat{G}$ is an abelian normal subgroup such that $\hat{G} / \ker(\pi_2) \simeq Im(\pi_2) = G$ is abelian; thus, $\hat{G}$ is metabelian, as we wanted to prove. 
\end{proof}

As we have seen that $\mathbb{F}_{p^n}^\alpha[G; \sigma]$ is a homomorphic image of $\mathbb{F}_p\hat{G}$, to compute the idempotents of $\mathbb{F}_{p^n}^\alpha[G; \sigma]$ it is enough to do it for $\mathbb{F}_p\hat{G}$ and compute its projection. 

\section{Subgroups of $\hat{G}$}
In this section, we will characterise the subgroups of $\hat{G}$, where $G$ will denote an abelian group, let $p$ be a prime number, and $\mathbb{F}_{p^n}$ be the finite field of $p^n$ elements, and let $\mathbb{F}_{p^n}^\alpha[G,\sigma]$ be a twisted skew group ring.

\begin{definition}
    Let $H\leq \hat{G}$ be a subgroup and  $\psi : \pi_2(H)\longrightarrow \mathbb{F}_q^*$ a set map such that $(\psi(g),g)\in H$ for all $g\in \pi_2(H)$. A Jineta of $H$ by $\psi$, is the set $Jin_\psi (H)=\{(\psi(g),g) ; g\in \pi_2(H)\}$. A Jineta of $H$ is a Jineta of $H$ by some $\psi$, and will be denoted as $Jin(\hat{H})$
\end{definition}

We will also need the following definition
\begin{definition}
     Let $H\leq \hat{G}$ be a subgroup. The Quag of $H$, denoted by $Q_{\hat{H}}$, is the set
     \begin{equation}
         Q_{\hat{H}} = \{a\in \mathbb{F}_q^* ; (a,1) \in H\}
     \end{equation}
\end{definition}

Note that the group $\mathbb{F}_q^*$ can be seen as a subgroup of $\hat{G}$ with its identification as $\mathbb{F}_q^*\times \{1\}$, and therefore, the product of a subgroup of $\mathbb{F}_q^*$ by a set of $\hat{G}$ is well defined.
\begin{proposition}
    Let $H\leq \hat{G}$ be a subgroup. Then $H=Q_{\hat{H}} \cdot Jin(\hat{H})$
\end{proposition}
\begin{proof}
    As $Q_{\hat{H}} , Jin(\hat{H}) \subseteq H$ we have that $Q_{\hat{H}} \cdot Jin(\hat{H}) \leq H$. To prove the other inclusion, fix $\psi $ such that $Jin(\hat{H}) = Jin_\psi (H)$. Let $(a,b)\in H$. If $(a,b)\in Jin(\hat{H})$ we have finished. If $(a,b) \not \in Jin(\hat{H})$, then there exist $(\psi(b),b) \in Jin(\hat{H}) \leq H$ and therefore $(a,b) (\psi(b),b)^{-1} \in H$. Then
    \begin{equation}
        (a,b) (\psi(b),b)^{-1}  = (a,b) \left(\frac{1}{(\psi(b)\alpha(b, b^{-1})^{b^{-1}})}, b^{-1} \right) = \left(\frac{a}{\psi(b)},1\right) \in H
    \end{equation}
    So $\left(\frac{a}{\psi(b)},1\right)  \in Q_{\hat{H}}$, and $\left(\frac{a}{\psi(b)},1\right)  \cdot (\psi(b),b) = (a,b)$ as we wanted to prove. 
\end{proof}

We can also produce the relation in the other direction.
\begin{lemma} \label{JinetaExtender}
    Let $H=Q_{\hat{H}} Jin_\psi(H) \leq \hat{G}$, and let $C$ be a set of generators of the cyclic components of $H$ . Then $\psi:H \longrightarrow \mathbb{F}_q^*$ is determined up to $Q_{\hat{H}}$ by the image of the elements of $C$. In addition, a map $\psi: C \longrightarrow \mathbb{F}_q^*$ with $(\psi(a),a)\in H   \quad \forall a\in C$ can be extended to a Jineta of $H$ if and only if for all $a,b\in C$, we have that
    \begin{align} 
        \frac{\psi(a)}{\psi(a)^b} \cdot \frac{\psi(b)^a}{\psi(b)} \cdot \frac{\alpha(a,b)}{\alpha(b,a)}& \in Q_{\hat{H}} \\
       \prod_{k=1}^{ord(a)-1}\alpha(a,a^k) \prod_{k=0}^{ord(a)-1} \psi(a)^{a^k}& \in Q_{\hat{H}} 
    \end{align}
    or equivalently
        \begin{align} \label{EqDef}
        \left(\frac{\psi(a)}{\psi(a)^b} \cdot \frac{\psi(b)^a}{\psi(b)} \cdot \frac{\alpha(a,b)}{\alpha(b,a)} \right) ^ { |Q_{\hat{H}}|} =& 1 \\
       \left(\prod_{k=1}^{ord(a)-1}\alpha(a,a^k) \prod_{k=0}^{ord(a)-1} \psi(a)^{a^k}\right) ^ { |Q_{\hat{H}}|} =&  1 
    \end{align}
\end{lemma}
\begin{proof}
    First of all, we have that $(\psi(a),a), (\psi(b),b) \in H$ so
    \begin{equation}
        (\psi(a),a), (\psi(b),b) = ( \psi(a) \psi(b)^a, ab) \in H
    \end{equation}
    As we have that $(\psi(ab),ab) \in H$, then by the definition of $Q_{\hat{H}}$, we have that $ \frac{\psi(ab)}{\psi(a) \psi(b)^a \alpha(a,b)} \in Q_{\hat{H}}$. The first results come from induction on the number of products. 

    For the second part of the proposition, we know that if $\psi$ gives a Jineta of $H$, then it must satisfy the previous identity. In addition, as $G$ is abelian, we have that $\psi(ab) = \psi(ba)$. Therefore, we have that 
    \begin{equation}
        \psi(a) \psi(b)^a \alpha(a,b) Q_{\hat{H}} = \psi(b) \psi(a)^b \alpha(b,a)Q_{\hat{H}}
    \end{equation}
 Which leads to 
    \begin{equation}
         \frac{\psi(a)}{\psi(a)^b} \cdot \frac{\psi(b)^a}{\psi(b)} \cdot \frac{\alpha(a,b)}{\alpha(b,a)}  \in Q_{\hat{H}}
    \end{equation}
    As $\mathbb{F}_q^*$ is a cyclic group, $Q_{\hat{H}}$ is the only subgroup of order $ |Q_{\hat{H}}|$, and it has all elements $x\in \mathbb{F}_q^*$ such that $x^{|Q_{\hat{H}}|}$. For the second identity, we must have that $\pi_2((\psi(a),a)^{ord(a)}) = a^{ord(a)} = 1$, so $(\psi(a),a)^{ord(a)}) \in Q_{\hat{H}}$. Therefore, we have that 
    \begin{equation}
        (\psi(a),a)^{ord(a)} = \prod_{k=0}^{ord(a)-1}(\psi(a),a) = \left(\prod_{k=1}^{ord(a)-1}\alpha(a,a^k) \prod_{k=0}^{ord(a)-1} \psi(a)^{a^k}, 1\right) \in Q_{\hat{H}} 
    \end{equation}
    Reasoning as in the previous case, we have that 
    \begin{equation}
        \left(\prod_{k=1}^{ord(a)-1}\alpha(a,a^k) \prod_{k=0}^{ord(a)-1} \psi(a)^{a^k}\right) ^ { |Q_{\hat{H}}|} =1
    \end{equation}

    Now, for the if part of the second statement, suppose that those two equations are satisfied. Then, all elements in $H$ can be written as a word in terms of $C$, $ a=\prod_{i=1}^n a_i$ and we can define
    
    \begin{align*}
    \psi^* :  H &  \longrightarrow \mathbb{F}_q^* / Q_{\hat{H}} \\
        a=\prod_{i=1}^n a_i &\longmapsto \psi^*\left( \prod_{i=1}^h a_i \right) = \prod_{i=1}^{n} \left( \psi(a_i)^{\prod_{j=1}^{i-1} a_j} \alpha\left(\prod_{j=1}^{i-1} a_j, a_i\right)\right) \quad \mod Q_{\hat{H}} 
    \end{align*}
    
    To check that this map is well-defined, it suffices to show that $\psi^*(AabB) = \psi^*(AbaB)$ for all $A, B \in G$ and $a, b \in C$. The first identity of equation (\ref{EqDef}) provide that the word of length $2$ are well defined. By induction, lets suppose that the words up to $n-1$ are well defined. Then, we have that
    
    \begin{align*}
        \psi^*(AabB) & \equiv_{Q_{\hat{H}}} \psi(A)\psi(a)^A \psi(b)^{Aa} \psi(B)^{Aa} \alpha(A,a)\alpha(Aa,b)\alpha(Aab,B) \\ 
        & \equiv_{Q_{\hat{H}}} \psi(A) \left(\psi(a) \psi(b)^{a}\right)^{A} \psi(B)^{Aab} \alpha(A,a)\alpha(Aa,b)\alpha(Aab,B) \\ 
        & \equiv_{Q_{\hat{H}}} \psi(A) \left(\psi(b) \psi(a)^{b} \frac{\alpha(b,a)}{\alpha(a,b)}\right)^{A} \psi(B)^{Aba} \alpha(A,a)\alpha(Aa,b)\alpha(Aab,B) \\ 
        & \equiv_{Q_{\hat{H}}} \psi(A) \psi(b)^A \psi(a)^{Ab}  \psi(B)^{Aba} \alpha(A,a)\alpha(Aa,b)\alpha(Aab,B) \left(\frac{\alpha(b,a)}{\alpha(a,b)}\right)^{A} \\ 
        & \equiv_{Q_{\hat{H}}} \psi(A) \psi(b)^A \psi(a)^{Ab}  \psi(B)^{Aba} \alpha(a,b)^A\alpha(A,ab)\alpha(Aab,B) \left(\frac{\alpha(b,a)}{\alpha(a,b)}\right)^{A}\\
        & \equiv_{Q_{\hat{H}}} \psi(A) \psi(b)^A \psi(a)^{Ab}  \psi(B)^{Aba} \alpha(a,b)^A\alpha(A,ba)\alpha(Aab,B)\left(\frac{\alpha(b,a)}{\alpha(a,b)}\right)^{A} \\ 
        & \equiv_{Q_{\hat{H}}} \psi(A) \psi(b)^A \psi(a)^{Ab}  \psi(B)^{Aba} \frac{\alpha(a,b)^A}{\alpha(b,a)^A}\alpha(A,b) \alpha(Ab,a)\alpha(Aab,B) \left(\frac{\alpha(b,a)}{\alpha(a,b)}\right)^{A}\\ 
        & \equiv_{Q_{\hat{H}}} \psi(A) \psi(b)^A \psi(a)^{Ab}  \psi(B)^{Aba} \alpha(A,b)\alpha(Ab,a)\alpha(Aba,B) \left(\frac{\alpha(a,b)}{\alpha(b,a)}\right)^{A} \\ 
        & \equiv_{Q_{\hat{H}}} \psi(AbaB)
    \end{align*}

    Therefore, we can define $\psi(a)$ as a representative of $\psi^*(a)$. Finally, it is trivial to see that $\{(\psi(a),a) ; a \in \pi_2(H)\} \subset H$ and therefore it define a Jineta for $H$
\end{proof}

\begin{proposition} \label{EcuacionesG'}
    Let $H\leq G$ and $Q\leq \mathbb{F}_q^*$ be a subgroup of $G$ and $\mathbb{F}_q^*$ respectively. Let $C$ be a set of generators of the cyclic components of $H$. If $\lambda$ is a generator of $\mathbb{F}_q^*$, we will denote $\mu : \mathbb{F}_q^* \longrightarrow \mathbb{Z}_{q-1}$ the homomorphism given by $\mu(\lambda)=1$ and extended multiplicatively. We will also denote by $\mu(a) = \mu(\lambda ^a)$.
    
     Then there exist $B \leq \hat{G}$ with $Q_B=Q, \pi_2(B) = H$ if and only if the following system of equations $\mod \frac{q-1}{|G|} $ on $\{\Psi(a) ; a\in C \}$ has a solution.
    \begin{align} \label{SysEqDadoGrupoandCylic}
        \Psi(a) (\mu(b)-1) + \Psi(b) ( \mu(a) -1) & \equiv  \mu(\alpha(a,b)) - \mu(\alpha(b,a))  &  \forall a,b\in C \\
        \Psi(a) \left( \frac{p^m-1}{\mu(a) -1} \right) & \equiv - \sum _{k=1}^{ord(a)-1} \mu(\alpha(a,a^k)) & \forall a\in C 
    \end{align}
    
\end{proposition}

\begin{proof}
 To prove this, let $\psi(a) = \lambda^{\Psi(a)}$. Then the first equation is the first equation of (\ref{EqDef}), while for the second one, note that
 \begin{equation}
      \frac{p^m-1}{\mu(a) -1} = 1 + \mu(a) + \mu(a)^2 + \cdots + \mu(a)^{ord(a)-1}
 \end{equation}
 So our new second equation is the second equation of (\ref{EqDef}). Therefore, $\psi$ can be extended from $C$ to $\langle C \rangle = H$, and the set $B=Q\cdot \{(\psi(a),a) ; a\in H\} $ is a group that satisfies our conditions. 
\end{proof}

Finally, we need a condition for the cyclicity of a quotient in $\hat{G}$. We have the following easy result.

\begin{lemma}
    Let $\hat{H}\leq \hat{G}$ be a subgroup of $\hat{G}$, let $H=Jin_\psi(\hat{H})$ and $Q=Q_{\hat{H}}$, and $\lambda$ the generator of $\mathbb{F}_q^*$. Then $\pi_2' : \hat{G} / \hat{H} \longrightarrow G/H$ given by $\pi_2((a,b)\hat{H}) = bH$ is an epimorphism of groups.
\end{lemma}
\begin{proof}
If we denote by $\pi_H $ the projection $G$ over $G/H$, then we have the diagram
    \[
\begin{tikzcd}[row sep=large, column sep=large]
\hat{G} \arrow[r, "\pi_2"] \arrow[rd, dashed, "\gamma"'] & G \arrow[d, "\pi_H"] \\
& G/H
\end{tikzcd}
\]

where $\gamma$ is an homomorphism from $\hat{G}$ to $G/H$. As $\hat{H} \subseteq\ker(\gamma)$, then $g$ induces a homomorphism of group $\gamma':\hat{G}/\hat{H} \longrightarrow G/H$, given by $\gamma'((a,b)\hat{H}) = bH$. Therefore $\pi_2 = \gamma'$ and we have the desired result.
\end{proof}

\begin{corollary}
        Let $\hat{H}\leq \hat{G}$ be a subgroup of $\hat{G}$, let $H=Jin_\psi(\hat{H})$ and $Q=Q_{\hat{H}}$, and $\lambda$ be the generator of $\mathbb{F}_q^*$.If $\hat{G}/ \hat{H}$ is cyclic, then $G/H$ is cyclic, and $\hat{G}/ \hat{H}$ is generated by $(a,t)\hat{H}$, where $t H$ is some generator of $G/H$. 
\end{corollary}
\begin{proof}
 First, suppose that $\hat{G}/\hat{H}$ is cyclic. Then $G/H =\pi'_2(\hat{G}/ \hat{H})$ is also cyclic. In addition, if $(a,t)\hat{H}$, clearly $ \pi_2'( (a,t) \hat{H}) =t H$ is some generator of $G/H$.

\end{proof}

\begin{lemma}
    Let $H\leq K$ abelian groups such that $K/H$ is cyclic. Then $K/H$ is generated by some element $t\in K$ or order $K/H$ 
\end{lemma}
\begin{proof}
    Decompose in cyclic components $K= \prod_{i,j} C_{p_i^{k_{i,j}}}$ with $p_i$ prime, $k_{i,j}\in \mathbb{N}$. Then $H=\prod_{i,j} C_{p_i^{h_{i,j}}}$ with $h_{i,j} \leq k_{i,j}$ and $h_{i,j} \not = k_{i,j}$ for at most one $j$ for each $i$. Let $\lambda=(\lambda_{i.j}) \in K$ given by $\lambda_{i,j}=1$ for all $i,j$ such that $h_{i,j}  = k_{i,j}$ and $\lambda_{i,j}$ of order $p^{k_{i,j}-h_{i,j}}$, which is possible as each component is cyclic. Then $\lambda$ has order $|K/H|$, and generate the quotient $K/H$.
\end{proof}

\begin{lemma}
      Let $\hat{H}\leq \hat{G}$ be a subgroup of $\hat{G}$ and let $tH$ be a generator of the cyclic group $G/H$. Then $\hat{G}/\hat{H}$ is cyclic if and only if it is generated by an element of the form $(a,t^k) \hat{H}$ for some $a\in \mathbb{F}_q^*, k \in \{0,\cdots, ord(t)-1\}$
\end{lemma}
\begin{proof}
    Suppose that $\hat{G}/\hat{H}$ is cyclic. Then it is generated by $(x,b) \hat{H}$. Apply $\pi_2'$ and we get that $bH= t^kH$ for some $k$, that is, $b= t^kh$ with $h\in H$. Then let $a=(\frac{x}{\psi(h)\alpha(h,t^k)})^{h^{-1}}$; therefore, we have that $[(a,t^k)] = [(\psi(h),h)][((\frac{x}{\psi(h)\alpha(h,t^k)})^{h^{-1}},t^k)]  = [(x,ht^k)] = [(x,b)]$
\end{proof}

\begin{corollary}
        Let $\hat{H}\leq \hat{G}$ be a subgroup of $\hat{G}$, let $H=Jin_\psi(\hat{H})$ and $Q=Q_{\hat{H}}$, and $\lambda$ the generator of $\mathbb{F}_q^*$.Then $\hat{G}/ \hat{H}$ is cyclic if and only if there exist $t\in G, a \in \mathbb{F}_q^*$ such that $t H$ is some generator of $G/H$ and $[(\lambda,1)] \in [(a,t)]$. Furthermore, such $t$ exist if and only if 
            \begin{equation} \label{eqCyclic}
          1 - \alpha_{|G/H |}'(t) - \mu( \alpha(t^{|G/H |},t^{-|G/H |}))-\Psi(t^{|G/H |})\mu(t^{|G/H |})= N_{|G/H |}(t) x  \mod{ \frac{q-1}{|Q|}}
    \end{equation}
    has a solution, in which case $a=\lambda^x$
\end{corollary}
\begin{proof}
    The only if is direct, and for the other implication we have that if $[(\lambda,1)] \in \langle [(a,t)] \rangle$ then all elements of the form $[(b,t^k)]$ can be written as $[(b,t^k)] = [(b/(a^{1+t+\cdots + t^{k-1}} \prod_{h=1}^{k-1}(t,t^h),1 ] [(a,t)]^k \in \langle [(a,t)] \rangle$

    Now, $[(\lambda,1)] \in \langle [(a,t)] \rangle$ if and only if $\exists k$ such that $(\lambda,1) = (a,t)^k (c,d)$ such that $(c,d) \in \hat{H}$. Therefore, we have that $k=|G/H |, d=t^{-k}$, and we have to solve $\lambda \in Q a^{N_k(t)} \alpha_k^*(t)\alpha(t^k,t^{-k})\psi(t^{-k})^{t^k}$. This is equivalent to solving the system
    \begin{equation}
        1 = N_k(t) \mu(a) + \alpha_k'(t)+ \mu( \alpha(t^k,t^{-k}))+\Psi(t^{-k})\mu(t^k) \mod{ \frac{q-1}{|Q|}}
    \end{equation}
\end{proof}

Finally, computing the Wedderburn decomposition requires the normalizer of a subgroup, as well as the group $E_{\hat{G}}(K/H)$ associated with each Shoda pair $(K,H)$. To this end, we establish thee straightforward results.

\begin{lemma}
     Let $H=Q_{\hat{H}} \cdot Jin_\psi(H) \leq \hat{G}$ be a subgroup. Then $\sigma(\pi_2(H)) = \left \langle \gamma \right \rangle$ for some $\gamma \in Aut(\mathbb{F}_q)$
\end{lemma}
\begin{proof}
    Note that $\sigma : G \longrightarrow Gal(\mathbb{F}_q,\mathbb{F}_p) \simeq C_m$, where $C_m$ is the cyclic group of $m$ elements. Therefore, $\sigma (H) \leq C_m$ is cyclic and generated by some element $\gamma \in Gal(\mathbb{F}_q,\mathbb{F}_p)$. 
\end{proof}

\begin{proposition}
    Let $\hat{H}=Q_{\hat{H}} Jin_{\psi}(\hat{H})\leq \hat{G}$ be a subgroup, $\pi_2(\hat{H}) =H$, and  $\sigma(H) = \langle \gamma_H\rangle$. Then the normalizer of $\hat{H}$ on $\hat{G}$, $N_{\hat{G}}(\hat{H})$, is given by $N_{\hat{G}}(\hat{H})=Q_N \cdot Jin_{\psi^*}(\hat{T})$, with $Q_N\leq \mathbb{F}_q^*, T=\pi_2(\hat{T}) \leq G$ characterided by
    \begin{equation}
        Q_N = \{x \in  \mathbb{F}_q^*,x^{\gamma_H-1} \leq Q_{\hat{H}} \}
    \end{equation}
    And $T$ is the maximal set of $\hat{G}$ such that there exist an $x$ satifiying 
    \begin{equation}
        x^{1-d}  \in Q_{H} \psi(d)^{1-b} \frac{\alpha(d,b)}{\alpha(b,d)}\quad \forall b \in H
    \end{equation}
    In which case, we can take $\psi^*(b)=x$ 

    or equivalently that the system of equations on $\delta(x)$
    \begin{equation}
        \delta(x)(1-\mu(d) = \Psi(d)(1-\mu(b)) + \mu(\alpha(d,b)) - \mu(\alpha(b,d)) \mod |Q_{\hat{H}}| \quad \forall b \in H
    \end{equation}
\end{proposition}
\begin{proof}
    Let $(a,b)\in N, (c,d) \in \hat{H}$. Then
    \begin{align*}
        (a,b)(c,d)(a,b)^{-1} & = (a,b)(c,d) \left( \frac{1}{a^{b^{-1}}\alpha(b,b^{-1})^{b^{-1}}} ,b^{-1}\right) \\ & = \left( \frac{a}{a^{bdb^{-1}}} c^{b}\frac{\alpha(b,d)\alpha(bd,b^{-1})}{\alpha(b,b^{-1})^{bdb^{-1}}}, bdb^{-1}\right) \\ & = \left( \frac{a}{a^{d}} c^{b}\frac{\alpha(b,d)\alpha(db,b^{-1})}{\alpha(b,b^{-1})^{d}}, d\right) \\ & = \left( \frac{a}{a^{d}} c^{b}\frac{\alpha(b,d)}{\alpha(d,b)}, d\right) 
    \end{align*}
    As $(a,b)$ is in the normaliser, this should be in $\hat{H}$, which is equivalent to
    \begin{equation}
        \frac{a}{a^{d}} \frac{c^{b}}{c}\frac{\alpha(b,d)}{\alpha(d,b)} \in Q_{H}
    \end{equation}
    Now, suppose that $b=1$. This is reduced to
    \begin{equation}
        \frac{a}{a^{d}} \in Q_{H}
    \end{equation}
   Since this holds for every $d \in H$, it is equivalent to $a a^{-\gamma_H} \in Q_{\hat{H}}$. In a group setting, this condition simplifies to $a^{\gamma_H - 1} \in Q_{\hat{H}}$, from which we conclude that $Q_N^{\gamma_H - 1} \subseteq Q_{\hat{H}}$. To prove that it is maximal, let $(x,1) \in \hat{G}$ such that $x^{\gamma_H-1} \in Q_{\hat{H}}$. Then $(x,1)(c,d)(\frac{1}{x},1) =  \left(\frac{x}{x^d}c,d \right)$. As $x^{\gamma_H-1} \in Q_{\hat{H}}$, in particular $x^{d-1} \in Q_{\hat{H}}$ and $\frac{x}{x^d}\in Q_{\hat{H}}$, so $\left(\frac{x}{x^d}c,d \right) = \left(\frac{x}{x^d},1) (c,d \right)\in H$ and $(x,1) \in N$.

    For $T$, we can take $a=\psi^*(b)$, and we get that
    \begin{equation}
        \psi^{*}(b)^{1-d} \in Q_{\hat{H}} c^{1-b} \frac{\alpha(d,b)}{\alpha(b,d)}
    \end{equation}
    As $c\in \psi(d) Q_N$ and $Q_N^{1-b} \subseteq Q_{\hat{H}}$, this can be reduced to
    \begin{equation}
        \psi^{*}(b)^{1-d} \in Q_{\hat{H}} \psi(d)^{1-b} \frac{\alpha(d,b)}{\alpha(b,d)}
    \end{equation}
    And therefore for all $b\in T$ there exist such $x$, given by $x=\psi^*(b)$. Finally, if for some $b$ there exist $x$ such that $ x^{1-d}  \in Q_{H} \psi(d)^{1-b} \frac{\alpha(d,b)}{\alpha(b,d)}\quad \forall b \in H$, then clearly $(x,b) \in N$ and $b\in T$ as we wanted to prove.
\end{proof}

Let $t\in G$ and $k\in \mathbb{N}$. We will denote by $N_k(t) = 1 + \sigma(t) + \sigma^2(t) + \cdots + \sigma^{k-1}(t)$ and $\alpha_k(t) = \prod_{i=1}^{k-1}\alpha(t,t^i)$. Now, to compute $E=E_{\hat{G}}(K/H)$ under the assumption that $K/H$ is cyclic, we have that
\begin{proposition}
    Let $\hat{H}, \hat{K} \leq \hat{G}$ such that $(\hat{K}, \hat{H})$ is a shoda pair. Then $E=Q_E Jin_{\psi^*}(E)$ is given by

    \begin{equation}
        Q_{E}=\left\{x \in \mathbb{F}_q^*; x^{1-t}  \in Q_{\hat{H}} \frac{c^{N_{3^k}(t)} \alpha_k(t)}{c}  \right\}
    \end{equation}
    and $\pi_2(E)$ is the set of elements such that the following system has a solution
    \begin{equation}
        x^{1-t}\in Q_{\hat{H}} \frac{c^{N_{3^k}(t)} \alpha_k(t)}{c^b} \frac{\alpha(t,b)}{\alpha(b,t)}
    \end{equation}
    for some $k\in \{1,\cdots, [K:H] \}$ such that $[K:H] | 3^k-1$. In that case, we can take $\psi^*(b)=x$.
\end{proposition}
\begin{proof}
    Let $(c, t) \hat{H} \in \hat{K}/ \hat{H}$ be a generator, with $tH$ a generator of $K/H$ with $t\in K$ of order $|K/H| = [K:H]$. Then, we have that
    \begin{equation}
        ((c, t) \hat{H}) ^{p^k} = (c^{N_{3^k}(t)} \alpha_k(t),t^{3^k} ) \hat{H}
    \end{equation}
    Therefore, if $(a,b) \in E$ we have that
    \begin{equation}
        (c^{N_{3^k}(t)} \alpha_k(t),t^{3^k} ) \hat{H} = (a,b)(c, t)(a,b)^{-1} \hat{H} = \left( \frac{a}{a^t}c^b \frac{\alpha(b,t)}{\alpha(t,b)} ,t\right)\hat{H} 
    \end{equation}
    As a result $t^{3^k} = t \mod H$, so we have that $[K:H] | 3^k-1$. Consequently, we have that
    \begin{equation}
        \frac{a}{a^t} \frac{c^b}{c^{N_{3^k}(t)} \alpha_k(t)} \frac{\alpha(b,t)}{\alpha(t,b)} \in Q_{\hat{H}}
    \end{equation}
    Finally, $Q_{\hat{H}}$ as well as in the previous proposition can be proved that it is determined by
    \begin{equation}
        Q_{E}=\left\{x \in \mathbb{F}_q^*; x^{1-t}  \in Q_{\hat{H}} \frac{c^{N_{3^k}(t)} \alpha_k(t)}{c}  \right\}
    \end{equation}
    and $\pi_2(E)$ can be characterised as the set of elements $b\in G$ such that the following system has a solution
    \begin{equation}
        x^{1-t}\in Q_{\hat{H}} \frac{c^{N_{3^k}(t)} \alpha_k(t)}{c^b} \frac{\alpha(t,b)}{\alpha(b,t)}
    \end{equation}
    for some $k\in \{1,\cdots, [K:H] \}$ such that $[K:H] | 3^k-1$
\end{proof}

\section{Shoda Pairs}
In this section, we will compute the idempotents for $\mathbb{F}_p \hat{G}$ that lead to the Wedderburn decomposition. To do so, we compute $\hat{H}'$ for all subgroups $\hat{H}$ of $\hat{G}$, as well as a maximal abelian subgroup of $\hat{G}$.

In order to apply Theorem \ref{metabelian}, we need to obtain a maximal abelian subgroup of $\hat{G}$.
\begin{definition}
    Let $Z_\alpha^\sigma(G)$ be the maximal subgroup of $\ker(\sigma) \cap Z(G)$ such that $\alpha|_{Z_\alpha^\sigma(G)}$ is a symmetric $2$-cocycle.
\end{definition}

\begin{proposition}
    The set $A=\mathbb{F}_q^* \times Z_\alpha^\sigma(G)$ is a maximal abelian subgroup of $\hat{G}$.
\end{proposition}
\begin{proof}
    It is a subgroup of $\hat{G}$ as $Z_\alpha^\sigma(G)$ is a subgroup of $\hat{G}$ and $Z_\alpha^\sigma(G) = \pi_2(A)$ To see that it is abelian, let $(a,b),(c,d)\in  A$. Then
    \begin{align*}
        (a,b)(c,d) = & (ac^b \alpha(b,d),bd) \\ = & (ac \alpha(b,d),bd) \\ = & (ca\alpha(d,b),db) \\ = & (ca^d \alpha(d,b), db) \\ = & (c,d) (a,b)
    \end{align*}

    Now we will prove that it is maximal. Suppose there exists an abelian subgroup $B\leq \hat{G}$ such that $A\leq B$. Then $\mathbb{F}_q^* \times \{1\}\leq A \leq B$, so $B=\mathbb{F}_q^* \times \pi_2(B)$.
    
      Take $a\in \mathbb{F}_q^*, b \in \pi_2(B)$, so $(1,b),(a,1) \in B$, and as it has to be commutative
    \begin{equation}
        (a,b) = (a,1)(1,b) = (1,b)(a,1) =  (a^b,b)
    \end{equation}
    so $a=a^b$ for all $b\in \pi_2(B)$ and therefore $B\leq \ker(\sigma)$. Now, for the $2$-cocycle, take $a,b\in \pi_2(B)$, then $(1,a),(1,b) \in B$, and as it has to be commutative:
    \begin{equation}
        ( \alpha(a,b), ab) = (1,a)(1,b) = (1,b)(1,a) = (\alpha(b,a),ba)) 
    \end{equation}
    so $\alpha|_B$ is symmetric. But as $A$ is maximal under this conditions, we have that $B=A$. 
\end{proof}

Now we will compute $\hat{H}'$ for all subgroups $\hat{H}$ of $\hat{G}$

\begin{proposition} \label{H'}
    Let $\hat{H}=Q_{\hat{H}} \cdot Jin_\psi(\hat{H}) \leq \hat{G}$ be a subgroup and $\sigma(\pi_2(\hat{H})) = \left \langle \gamma \right \rangle$. Then 
    \begin{equation}
        \hat{H}' = \left\langle \frac{a}{\gamma(a)}, \frac{\psi(b)}{\psi(b)^d}\frac{\psi(d)^b}{\psi(d)} \frac{\alpha(b,d) }{\alpha(d,b)}; a\in  Q_{\hat{H}}, b,d \in \pi_2(\hat{H}) \right\rangle \times {1}
    \end{equation}
\end{proposition}

\begin{proof}
We will study $\hat{H}'$.

For all $(a,b),(c,d) \in \hat{H}$
    \begin{align*}
        [(a,b), (c,d)] & = (a,b) (c,d) (a,b)^{-1} (c,d)^{-1} \\
         & = (a,b) (c,d) \left(\frac{1}{a\alpha(b,b^{-1})}, b^{-1} \right)\left(\frac{1}{c\alpha(d,d^{-1})}, d^{-1} \right) \\
         & =\left(ac^b \alpha(b,d), bd\right)\left(\frac{1}{(a\alpha(b,b^{-1}))^{b^{-1}}}, b^{-1} \right)\left(\frac{1}{(c\alpha(d,d^{-1}))^{d^{-1}}}, d^{-1} \right) \\
         & = \left( \frac{ac^b \alpha(b,d) \alpha(bd,b^{-1})}{a^{d}\alpha(b,b^{-1})^{d}},bdb^{-1} \right) \left(\frac{1}{(c\alpha(d,d^{-1}))^{d^{-1}}}, d^{-1} \right)  \\
         & = \left( \frac{ac^b \alpha(b,d)\alpha(bd,b^{-1}) \alpha(bdb^{-1}),d^{-1})}{a^{d}\alpha(b,b^{-1})^{d} c \alpha(d,d^{-1})} ,bdb^{-1} d^{-1}\right) \\
         & = \left( \frac{a}{a^d} \cdot \frac{c^b}{c} \cdot \frac{\alpha(b,d)\alpha(bd,b^{-1}) }{\alpha(b,b^{-1})^{d}} ,1\right) \\
          & = \left( \frac{a}{a^d} \cdot \frac{c^b}{c} \cdot \frac{\alpha(b,d) }{\alpha(d,b)} ,1\right) 
    \end{align*} 

 If we take $b=1$ we get that $\left(\frac{a}{a^d},1\right)\in \hat{H}'$ for all $a\in Q_{\hat{H}}, d \in \pi_2(\hat{H})$. In particular, we have $\left(\frac{a}{\gamma(a)},1\right)\in \hat{H}'$. For the second element, it is enough to take $(\psi(b),b),( \psi(d),d) \in \hat{H}$. 

To prove that this generates the whole group, first note that  $\sigma(d) = \gamma^j$ for some $j\in \mathbb{N}$. In addition, $\gamma(a)=a^{p^h}$ for some $h\in \mathbb{N}$. Therefore, we have that

\begin{equation}
\frac{a}{\gamma(a)} \cdot \left(\frac{a}{\gamma(a)} \right)^{p^h} \cdot \left(\frac{a}{\gamma(a)} \right)^{p^{2h}} \cdots \left(\frac{a}{\gamma(a)} \right)^{p^{jh}} = \frac{a}{\gamma(a)} \frac{\gamma(a)}{\gamma^2(a)} \cdots \frac{\gamma^{j-1}(a)}{\gamma^j(a)} = \frac{a}{\gamma^j(a)} = \frac{a}{a^d}
\end{equation}

So $\left(\frac{a}{a^d},1 \right) \in \left \langle \left(\frac{c}{\gamma(c)}, 1 \right) ; c\in \mathbb{F}_q \right \rangle$. Now, we have that 

\begin{align*}
        [(a,b), (c,d)] 
         & = \left( \frac{a}{a^d} \cdot \frac{c^b}{c} \cdot \frac{\alpha(b,d) }{\alpha(d,b)},1\right)  \\
         & = \left( \frac{a/\psi(b)}{(a/\psi(b))^d} \cdot \frac{(c/\psi(d))^b}{c/\psi(d)} \cdot \frac{\psi(b)}{\psi(b)^d} \frac{\psi(d)^b}{\psi(d)}\frac{\alpha(b,d) }{\alpha(d,b)},1\right)  \\
         & = \left( \frac{a/\psi(b)}{(a/\psi(b))^d} , 1\right)\left( \frac{(c/\psi(d))^b}{c/\psi(d)}  , 1\right) \left(\frac{\psi(b)}{\psi(b)^d} \frac{\psi(d)^b}{\psi(d)}\frac{\alpha(b,d) }{\alpha(d,b)},1 \right)
         \\
         & = \left( \frac{a/\psi(b)}{(a/\psi(b))^d} , 1\right)\left( \frac{(c/\psi(d))}{c/\psi(d)^b}  , 1\right)^{-1} \left(\frac{\psi(b)}{\psi(b)^d} \frac{\psi(d)^b}{\psi(d)}\frac{\alpha(b,d) }{\alpha(d,b)},1 \right)
    \end{align*} 
\end{proof}

As a result, we have that 
\begin{corollary} 
    Let $K=\mathbb{F}_p^* \cdot Jin(\hat{K}) \leq \hat{G}$, then 
     \begin{equation}
        K' = \left\langle \frac{a}{\gamma_K(a)}, \frac{\alpha(b,d) }{\alpha(d,b)}; a\in  \mathbb{F}_q^*, b,d \in \pi_2(K) \right\rangle \times {1}
    \end{equation}
Where $\sigma(K) = \langle \gamma_K\rangle$
\end{corollary}
\begin{proof}
    A direct application of proposition \ref{H'}.
\end{proof}

Finally, we have the following corollary
\begin{corollary}
        Let $\hat{H}=Q_{\hat{H}} Jin_\psi(\hat{H}) \leq \hat{G}$ and $\hat{K}= \mathbb{F}_q^* \cdot Jin(\hat{H}) \leq \hat{G}$ such that $\hat{K}'\leq \hat{H} \leq \hat{K}$, and let $C$ be a set of generators of the cyclic components of $H=\pi_2(\hat{H})$. Then the map $\psi: C \longrightarrow \mathbb{F}_q^*$ with $(\psi(a),a)\in \hat{H}   \quad \forall a\in C$ can be extended to a Jineta of $\hat{H}$ if and only if for all $a\in C$, we have that
    \begin{align} 
       \left(\prod_{k=1}^{ord(a)-1}\alpha(a,a^k) \prod_{k=0}^{ord(a)-1} \psi(a)^{a^k}\right) ^ { |Q_{\hat{H}}|} =&  1 
    \end{align}
\end{corollary}
\begin{proof}
    From \ref{JinetaExtender}, we have that $\psi(a)$ must satisfy
        \begin{align} 
        \left(\frac{\psi(a)}{\psi(a)^b} \cdot \frac{\psi(b)^a}{\psi(b)} \cdot \frac{\alpha(a,b)}{\alpha(b,a)} \right) ^ { |Q_{\hat{H}}|} =& 1 \\
       \left(\prod_{k=1}^{ord(a)-1}\alpha(a,a^k) \prod_{k=0}^{ord(a)-1} \psi(a)^{a^k}\right) ^ { |Q_{\hat{H}}|} =&  1 
    \end{align}
    however, as $\hat{K}'\leq \hat{H}$, we have that $\frac{\psi(a)}{\psi(a)^b} ,\frac{\alpha(a,b)}{\alpha(b,a)} \in Q_{\hat{H}} $, so the first condition is always true.
\end{proof}

We will focus on how to solve the previous equation. To do so, we will assume that $\alpha$ is given in terms of a generator $\lambda$ of $\mathbb{F}_q^*$. Therefore, if we denote by $\mu\left(\prod_{k=1}^{ord(a)-1}\alpha(a,a^k) \right) = \alpha'(a)$ we have that $\prod_{k=1}^{ord(a)-1}\alpha(a,a^k) = \lambda^{\alpha'(a)}$. In addition, let $N(a) = 1 + \sigma(a) + \cdot + \sigma^{ord(a)-1}(a)$, then the previous equation is equivalent to
\begin{align*}
    \lambda^{\alpha'(a)} \left(\lambda^{\Psi(a)}\right) ^{N(a)} \in Q_{\hat{H}}
\end{align*}

Which is equivalent to
\begin{equation*}
   (\alpha'(a) + \Psi(a)N(a)) |Q_{\hat{H}}| \equiv_{q-1} 0 
\end{equation*}
As $|Q_{\hat{H}}| $ divides $q-1$, we have that 

\begin{equation*}
   \Psi(a)N(a) \equiv_{q-1 /  |Q_{\hat{H}}| } - \alpha'(a)
\end{equation*}

Therefore, we have proven the following result
\begin{proposition}
 Let $\alpha:G\longrightarrow \mathbb{Z}_{q-1}$ be given by $\prod_{k=1}^{ord(a)-1}\alpha(a,a^k) = \lambda^{\alpha'(a)}$, and let $N(a) = 1 + \sigma(a) + \cdot + \sigma^{ord(a)-1}(a)$. Let $\hat{H}=Q_{\hat{H}} Jin(\hat{H}) \leq \hat{G}$ and let $v(a)= gcd(N(a),q-1 /  |Q_{\hat{H}}| )$, $N(a)/v(a)=c(a)$, and let $C$ be a set of generators of $H$. Then $\psi:C \longrightarrow \mathbb{F}_{q-1}$ given by $\psi(a) = \lambda^{\Psi(a)}$ can be extended to a Jineta for $H$ if and only if
    \begin{equation*}
   \Psi(a)N(a) \equiv_{q-1 /  |Q_{\hat{H}}| } - \alpha'(a)
\end{equation*}

\end{proposition}

Now, as we want a pair of groups $(K,H)$ where $A\leq K$, we have that $K=\mathbb{F}_p^* \cdot Jin(\hat{K})$, and we can take $Jin(\hat{K})=\{1\} \times  \pi_2(K)$, with $Z_\alpha^\sigma(G) \leq \pi_2(K)$. In addition, based on the last proposition
    \begin{equation}
        K' = \left\langle \frac{a}{\gamma_K(a)}, \frac{\alpha(b,d) }{\alpha(d,b)}; a\in  \mathbb{F}_q^*, b,d \in \pi_2(K) \right\rangle \times {1}
    \end{equation}
Where $\sigma(K) = \langle \gamma_K\rangle$. As we have that
\begin{equation}
    H' \leq K' \leq H \leq K
\end{equation}
Then given a subgroup $H=Q_{\hat{H}} Jin_\psi(H)$, we must have $Q_{K'} \subset Q_{\hat{H}}$ and therefore

\begin{equation}
   \frac{\alpha(b,d)}{\alpha(d,b)} \in Q_{\hat{H}} \quad \forall a\in \mathbb{F}_q, b,d \in \pi_2(H)
\end{equation}

So we are under the conditions of proposition (\ref{EcuacionesG'}) and we can construct the following algorithm.

\begin{proposition}
    Let $G$ be an abelian group, $\mathbb{F}_q$ a finite field with $q=p^n$, $\sigma : G \longrightarrow Aut(\mathbb{F}_q)$ an action of $G$ over the field, and let  $\alpha\in Z^2(G,\mathbb{F}_q^*)$ be a $2-$cocycle. Then all the strongly Shoda pair are given by the following algorithm,
     \begin{algorithm}[H]
Let $G$ be an abelian group decomposed in its cyclic components.
    \begin{enumerate}
        \item Initialize the empty sets $A_1, A_2, A_3,A_4,A_5,T$
        \item Compute $Z_\alpha^\sigma(G)$
        \item Compute $K\leq G$ such that $Z_\alpha^\sigma(G) \leq K$ and store it in $A_1$
        \item For each $K\in J$ compute $Q_{K'}$ and store $(Q_{K'},K)$ in $A_2$. Store all possible values$\zeta(K)=|Q_{K'}|$ in the set $T$.
        \item For each $C_m \leq \mathbb{F}_q^*$ eliminate $(Q_{K'},K)$ from $A_2$ such that they are not maximal respect to $Q_{K'}=C_m$ 
        \item For each $(C_{\zeta(K)},K)$ in $A_2$ took $H\leq K$ such that $K/H$ is cyclic, $C_t$ with $\zeta(K)\leq t < h$ for all $\zeta(K)<h \in T$. Keep $(C_t, H, K)$ in the set $A_3$.
        \item For each $(C_t,H,K)$ compute $\alpha'$, $N$ for each generator of all $H$ and keep it $S$.  For each  $(C_t,H,K)$ try to solve the system (\ref{EcuacionesG'}), compute all possible $\psi$ and keep $(C_t\cdot Jin_\psi(H),J)$ them $A_4$
        \item For each $(A,B)\in A_4$ take a generator of $\pi_2(B)/ \pi_2(A)$ and try to solve (\ref{eqCyclic}). Eliminate those without solution. 
        \item Return $A_4$
    \end{enumerate}
\end{algorithm}
\end{proposition}

\section{Example}
In this section, we will show how this algorithm actually produces the idempotents of a twisted skew group algebra. 

Let $G=C_2 \times C_2 \times C_5$ and let $\mathbb{F}_9$ the finite field of $9$ elements. Let $\sigma: G \longrightarrow Aut(\mathbb{F}_9)$ given by $\sigma(a,b,c) (f) = f^{3^{a+b}}$ and the $2-$cocycle
\begin{align*}
    \alpha: G \times G \longrightarrow \mathbb{F}_{9}^* \\
    (a,b,c) , (d,e,f) \longmapsto 2^{ad}
\end{align*}

Let $\lambda = x + 2$ be a generator of the multiplicative group $\mathbb{F}_{9}^\times$, defined via the irreducible polynomial $x^2 + 1 \in \mathbb{F}_3[x]$. Since $2 = \lambda^4$, we can express $\alpha$ in terms of this generator. Furthermore, $\lambda$ satisfies the minimal polynomial relation $\lambda^2 + 2\lambda + 2 = 0$ over $\mathbb{F}_3$.

\begin{enumerate}
    \item Let $A_1, A_2, A_3, A_4, S,T$ empty sets.

\item First, we need to compute  $Z_\alpha^\sigma(G)$. To do so, we have that $\ker(\sigma) \cap Z(G)= \ker (\sigma) = \langle (1,1) \rangle\times C_5$. In addition, $\alpha$ is symmetric, so we have that $Z_\alpha^\sigma(G) =\langle (1,1) \rangle \times C_3$.

\item As a result, the lattice of $G$ with $Z_\alpha^\sigma(G)$ as a subgroup is given by

\begin{center}
    
\begin{tikzpicture}[node distance=2cm, main node/.style={circle, fill=white, draw, minimum size=2em, inner sep=0pt}]

    \node (G) at (0,2) {$C_2 \times C_2 \times C_5$};
    
    \node (E) at (0,0) {$\langle (1,1) \rangle \times C_5$};

    \draw (G) -- (E);
\end{tikzpicture}

\end{center}

\item We compute the associated table \ref{Tabla1} of $Q_{K'}$. To do so, we identify $\mathbb{F}_9^*$ with $C_{8}$

\begin{table} \label{Tabla1}
\caption{Associated table of $Q_{K'}$ for different subgroups $K$.}\label{tbl1}
\begin{tabular}{cc}
\toprule
$K$ & $Q_{K'}$ \\ 
\midrule
$C_2 \times C_2 \times C_5$ & $C_4$ \\
$\langle (1,1) \rangle \times C_5$ & $\{0\}$ \\
\bottomrule
\end{tabular}
\end{table}

So we have $T=\{1,4\}$. 

\item As we eliminate all the subgroup that are not maximal with respect each element in $T$, we have $A_2=\{ (\{0\}, \langle (1,1) \rangle\times C_5), ( C_4, C_2 \times C_2 \times C_5 ) \}$

\item We have the sets 

\begin{align*}
    A_3 = \{   & (C_8,\langle (1,0) \rangle \times C_5, \hat{G}))\\
    & (C_8,\langle (0,1) \rangle \times C_5, \hat{G}))\\
    & (C_8,\langle (1,1) \rangle \times C_5, \hat{G})) \\
     & (C_8, C_2 \times C_2 \times \{0\}, \hat{G}))\\
    & (C_8, \langle (1,1) \rangle \times \{0\}, \hat{G}))\\   
    & (C_8,\langle (1,0) \rangle \times \{0\}, \hat{G}))\\
    & (C_8,\langle (0,1) \rangle \times \{0\}, \hat{G}))\\
    & (C_4,\langle (1,0) \rangle \times C_5, \hat{G}))\\
    & (C_4,\langle (0,1) \rangle \times C_5, \hat{G}))\\
    & (C_4,\langle (1,1) \rangle \times C_5, \hat{G}))\\
    & (C_4, C_2 \times C_2 \times \{0\}, \hat{G}))\\
        & (C_4, C_2 \times C_2 \times C_5, \hat{G}))\\
        & (C_4, \langle (1,1) \rangle \times \{0\}, \hat{G}))\\
        & (C_4,\langle (1,0) \rangle \times \{0\}, \hat{G}))\\
    & (C_4,\langle (0,1) \rangle \times \{0\}, \hat{G}))\\
    & (0, \{0\} \times \{0\} \times \{0\}, Z_\alpha^\sigma(G))\\
    & (0, \{0\} \times \{0\} \times C_5, Z_\alpha^\sigma(G))\\
    & (0, \langle (1,1) \rangle \times \{0\}, Z_\alpha^\sigma(G))\\
    & (0, \langle (1,1) \rangle \times C_5, Z_\alpha^\sigma(G))\\
    & (C_2, \{0\} \times \{0\} \times \{0\}, Z_\alpha^\sigma(G))\\
    & (C_2, \{0\} \times \{0\} \times C_5, Z_\alpha^\sigma(G))\\
    & (C_2, \langle (1,1) \rangle \times \{0\}, Z_\alpha^\sigma(G))\\
    & (C_2, \langle (1,1) \rangle \times C_5, Z_\alpha^\sigma(G))\\
     \}
\end{align*}

    \item We compute $\alpha^*, \alpha', N$ for all generators in table \ref{Table2}.
    \begin{table}
\caption{Values of $\alpha$, $\alpha'$ and $N$ for each generator $C$.}\label{Table2}
\begin{tabular}{cccc}
\toprule
$C$ & $\alpha$ & $\alpha'$ & $N$ \\ 
\midrule
$(1,0,0)$ & $2$ & $4$ & $4$ \\
$(0,1,0)$ & $1$ & $0$ & $4$ \\
$(1,1,0)$ & $2$ & $4$ & $2$ \\
$(0,0,1)$ & $1$ & $0$ & $5$ \\
\bottomrule
\end{tabular}
\end{table}

     \item We compute complete the table \ref{Table3} with $\Psi$. As for $C_8=\mathbb{Z}_8$ all values are possible, we only solve it for $C_4=2\mathbb{Z}_8$, $C_2=4\mathbb{Z}_8$. 
        \begin{table} 
\caption{Values of $\alpha$, $\alpha'$, $N$, and image ideals under $\Psi$ associated with each generator $C$.}\label{Table3}
\begin{tabular}{ccccccc}
\toprule
$C$ & $\alpha$ & $\alpha'$ & $N$ & $\Psi (2\mathbb{Z}_8)$ & $\Psi (4\mathbb{Z}_8)$ & $\Psi (0)$ \\ 
\midrule
$(1,0,0)$ & $2$ & $4$ & $4$ & $\mathbb{Z}_8$  &                &                \\
$(0,1,0)$ & $1$ & $0$ & $4$ & $\mathbb{Z}_8$  &                &                \\
$(1,1,0)$ & $2$ & $4$ & $2$ & $\mathbb{Z}_8$  & $2\mathbb{Z}_8$ & $2\mathbb{Z}_8$ \\
$(0,0,1)$ & $1$ & $0$ & $5$ & $2\mathbb{Z}_8$ & $4\mathbb{Z}_8$ & $0$             \\
\bottomrule
\end{tabular}
\end{table}

    Note that $\psi(x)^2\mathbb{Z}_8$ generate all the same group in the case of $2\mathbb{Z}_8$, so we can count how many subgroup we have, and represent it in table \ref{Table4} and table \ref{Table5}.

\begin{table} 
\caption{Number of pairs $(\hat{H}, \hat{G})$ satisfying $\pi_2(\hat{H}) = H$ for each subgroup $H$.}\label{Table4}
\begin{tabular}{cc}
\toprule
$H$ & Number of $(\hat{H}, \hat{G})$ with $\pi_2(\hat{H}) = H$ \\ 
\midrule
$\langle (1,0) \rangle \times C_5$    & 3 \\
$\langle (0,1) \rangle \times C_5$    & 3 \\
$\langle (1,1) \rangle \times C_5$    & 3 \\
$\langle (1,1) \rangle \times \{0\}$  & 3 \\
$C_2 \times C_2 \times \{0\}$         & 5 \\
$C_2 \times C_2 \times C_5$           & 5 \\
\bottomrule
\end{tabular}
\end{table}

\begin{table} 
\caption{Number of pairs $(\hat{H}, A)$ satisfying $\pi_2(\hat{H}) = H$ for each subgroup $H$.}\label{Table5}
\begin{tabular}{cc}
\toprule
$H$ & Number of $(\hat{H}, A)$ with $\pi_2(\hat{H}) = H$ \\ 
\midrule
$\{0\} \times \{0\} \times C_5$ & 2 \\
$\{0\} \times \{0\} \times \{0\}$ & 2 \\
$\langle (1,1) \rangle \times \{0\}$ & 6 \\
$\langle (1,1) \rangle \times C_5$ & 6 \\
\bottomrule
\end{tabular}
\end{table}

    \item As for $C_8$ we have the cyclicity condition trivially, we will work with $C_4$. In addition, the cyclicity is also trivial for the subgroups of $Z_\alpha^\sigma(G)$ as it is an abelian group and it is direct to compute wether or not the quotient will be cyclic. Note that we can always assume $\Psi(1)=1$. We have the table \ref{table6}
\begin{table}
\caption{Modular equations in $2\mathbb{Z}_8$ and cyclicity determination for each subgroup $H$.}\label{table6}
\begin{tabular}{ccc}
\toprule
$H$ & Equation $2\mathbb{Z}_8$ & Cyclic \\ 
\midrule
$\langle (0,1) \rangle \times C_5$ & $-7 \equiv 4x \pmod{2}$ & No \\
$\langle (1,0) \rangle \times C_5$ & $1 \equiv 4x \pmod{2}$  & No \\
$\langle (1,1) \rangle \times C_5$ & $-7 \equiv 2x \pmod{2}$ & No \\
$C_2 \times C_2 \times \{0\}$      & $1 \equiv 5x \pmod{2}$  & Yes \\
$C_2 \times C_2 \times C_5$        & $-$                     & Yes \\
\bottomrule
\end{tabular}
\end{table}

    \item the Shoda pair are given by

    \begin{align*}
    A_3 = \{   & ( \{0\} \cdot Jin_{\psi_{1}}(\langle (1,1) \rangle \times \{0\} ) , A )  \\
   & ( C_{2} \cdot Jin_{\psi_{2}}(\langle (1,1) \rangle \times \{0\} ) , A )  \\
   & ( C_{2} \cdot Jin_{\psi_{3}}(\langle (1,1) \rangle \times \{0\} ) , A )  \\
   & \{0\} \cdot Jin_{\psi_{4}}(\langle (1,1) \rangle \times C_5) , A )  \\
   & ( C_{2} \cdot Jin_{\psi_{5}}(\langle (1,1) \rangle \times C_5) , A )  \\
   & ( C_{2} \cdot Jin_{\psi_{6}}(\langle (1,1) \rangle \times C_5) , A )  \\
   & ( C_{4} \cdot Jin_{\psi_{7}}(C_2 \times C_2 \times \{0\} ) , \hat{G} ) \\
   & ( C_{8} \cdot (\langle (1,1) \rangle \times \{0\} ) , \hat{G} ) \\
   & ( C_{4} \cdot Jin_{\psi_{9}}(C_2 \times C_2 \times \{0\} ) , \hat{G} ) \\
   & ( C_{8} \cdot(C_2 \times \{0\}  \times \{0\} ) , \hat{G} ) \\
   & ( C_{4} \cdot Jin_{\psi_{11}}(C_2 \times C_2 \times \{0\} ) , \hat{G} ) \\
   & ( C_{8} \cdot(\{0\} \times C_2 \times \{0\} ) , \hat{G} ) \\
   & ( C_{4} \cdot Jin_{\psi_{13}}(C_2 \times C_2 \times \{0\} ) , \hat{G} ) \\
   & ( C_{8} \cdot(C_2 \times C_2 \times \{0\} ) , \hat{G} ) \\
   & ( C_{4} \cdot Jin_{\psi_{15}}(C_2 \times C_2 \times C_5) , \hat{G} ) \\
   & ( C_{8} \cdot(\langle (1,1) \rangle \times C_5) , \hat{G} ) \\
   & ( C_{4} \cdot Jin_{\psi_{17}}(C_2 \times C_2 \times C_5) , \hat{G} ) \\
   & ( C_{8} \cdot(C_2 \times \{0\}  \times C_5) , \hat{G} ) \\
   & ( C_{4} \cdot Jin_{\psi_{19}}(C_2 \times C_2 \times C_5) , \hat{G} ) \\
   & ( C_{8} \cdot (\{0\} \times C_2 \times C_5) , \hat{G} ) \\
   & ( C_{4} \cdot Jin_{\psi_{21}}(C_2 \times C_2 \times C_5) , \hat{G} ) \\
   & ( C_{8} \cdot (C_2 \times C_2 \times C_5) , \hat{G} ) \\
    \}
\end{align*}
\end{enumerate}

Now, we will compute the Wedderburn decomposition, where we denote $N=N_{\hat{G}}(H) \cap N_{\hat{G}}(K) $ and obtain table \ref{table7}.

\begin{table} 
\caption{Normalizer and $E_{\hat{G}}(K/H)$ for each pair}\label{table7}
\begin{tabular}{cccc}
    \toprule
    $H$ & $K$ & $N $ &  $E_{\hat{G}}(K/H)$ \\ 
    \midrule
    $\{0\} \cdot Jin_{\psi_{1}}(\langle (1,1) \rangle \times \{0\} ) $,   &
 $A$  &
     $C_{8} \cdot (\langle (1,1) \rangle \times C_5) $,   &
     $C_{8} \cdot (\langle (1,1) \rangle \times C_5) $,  \\
     $C_{2} \cdot Jin_{\psi_{2}}(\langle (1,1) \rangle \times \{0\} ) $,   &
 $A$  &
     $\hat{G} $,   &
     $C_{8} \cdot (\langle (1,1) \rangle \times C_5) $,  \\
     $C_{2} \cdot Jin_{\psi_{3}}(\langle (1,1) \rangle \times \{0\} ) $,   &
 $A$  &
     $\hat{G} $,   &
     $C_{8} \cdot (\langle (1,1) \rangle \times C_5) $,  \\
     $\{0\} \cdot Jin_{\psi_{4}}(\langle (1,1) \rangle \times C_5) $,   &
 $A$  &
     $C_{8} \cdot (\langle (1,1) \rangle \times C_5) $,   &
     $C_{8} \cdot (\langle (1,1) \rangle \times C_5) $,  \\
     $C_{2} \cdot Jin_{\psi_{5}}(\langle (1,1) \rangle \times C_5) $,   &
 $A$  &
     $\hat{G} $,   &
     $\hat{G} $,  \\
     $C_{2} \cdot Jin_{\psi_{6}}(\langle (1,1) \rangle \times C_5) $,   &
 $A$  &
     $\hat{G} $,   &
     $\hat{G} $,  \\
     $C_{4} \cdot Jin_{\psi_{7}}(C_2 \times C_2 \times \{0\} ) $,   &
 $\hat{G}$  &
     $\hat{G} $,   &
     $\hat{G} $,  \\
     $C_{8} \cdot (\langle (1,1) \rangle \times \{0\} ) $,   &
 $\hat{G}$  &
     $\hat{G} $,   &
     $\hat{G} $,  \\
     $C_{4} \cdot Jin_{\psi_{9}}(C_2 \times C_2 \times \{0\} ) $,   &
 $\hat{G}$  &
     $\hat{G} $,   &
     $\hat{G} $,  \\
     $C_{8} \cdot(C_2 \times \{0\}  \times \{0\} ) $,   &
 $\hat{G}$  &
     $\hat{G} $,   &
     $\hat{G} $,  \\
     $C_{4} \cdot Jin_{\psi_{11}}(C_2 \times C_2 \times \{0\} ) $,   &
 $\hat{G}$  &
     $\hat{G} $,   &
     $\hat{G} $,  \\
     $C_{8} \cdot(\{0\} \times C_2 \times \{0\} ) $,   &
 $\hat{G}$  &
     $\hat{G} $,   &
     $\hat{G} $,  \\
     $C_{4} \cdot Jin_{\psi_{13}}(C_2 \times C_2 \times \{0\} ) $,   &
 $\hat{G}$  &
     $\hat{G} $,   &
     $\hat{G} $,  \\
     $C_{8} \cdot(C_2 \times C_2 \times \{0\} ) $,   &
 $\hat{G}$  &
     $\hat{G} $,   &
     $\hat{G} $,  \\
     $C_{4} \cdot Jin_{\psi_{15}}(C_2 \times C_2 \times C_5) $,   &
 $\hat{G}$  &
     $\hat{G} $,   &
     $\hat{G} $,  \\
     $C_{8} \cdot(\langle (1,1) \rangle \times C_5) $,   &
 $\hat{G}$  &
     $\hat{G} $,   &
     $\hat{G} $,  \\
     $C_{4} \cdot Jin_{\psi_{17}}(C_2 \times C_2 \times C_5) $,   &
 $\hat{G}$  &
     $\hat{G} $,   &
     $\hat{G} $,  \\
     $C_{8} \cdot(C_2 \times \{0\}  \times C_5) $,   &
 $\hat{G}$  &
     $\hat{G} $,   &
     $\hat{G} $,  \\
     $C_{4} \cdot Jin_{\psi_{19}}(C_2 \times C_2 \times C_5) $,   &
 $\hat{G}$  &
     $\hat{G} $,   &
     $\hat{G} $,  \\
     $C_{8} \cdot (\{0\} \times C_2 \times C_5) $,   &
 $\hat{G}$  &
     $\hat{G} $,   &
     $\hat{G} $,  \\
     $C_{4} \cdot Jin_{\psi_{21}}(C_2 \times C_2 \times C_5) $,   &
 $\hat{G}$  &
     $\hat{G} $,   &
     $\hat{G} $,  \\
     $C_{8} \cdot (C_2 \times C_2 \times C_5) $,   &
 $\hat{G}$  &
     $\hat{G} $,   &
     $\hat{G} $,  \\
     \bottomrule
 \end{tabular}
\end{table}

So we have table \ref{table8}, and after counting the number of orbits of each component, lead to

\begin{table}
\caption{Parameters of the Wedderburn decomposition}\label{table8}
\begin{tabular}{ccccc}
    \toprule
    $(H,K)$ & $ [\hat{G}:K] $ & $[K : H ]$ & ord of $q$ $(3)$ module $[K : H ]$ & $[E:K]$ \\ 
    \hline
     $C_{1} \cdot Jin_{\psi_{1}}(\langle (1,1) \rangle \times \{0\} ) $,   & $
2$ & $
40$ & $
4$ & $
1$ \\
    $C_{2} \cdot Jin_{\psi_{2}}(\langle (1,1) \rangle \times \{0\} ) $,   & $
2$ & $
20$ & $
4$ & $
1$ \\
    $C_{2} \cdot Jin_{\psi_{3}}(\langle (1,1) \rangle \times \{0\} ) $,   & $
2$ & $
20$ & $
4$ & $
1$ \\
    $C_{1} \cdot Jin_{\psi_{4}}(\langle (1,1) \rangle \times C_5) $,   & $
2$ & $
8$ & $
2$ & $
1$ \\
    $C_{2} \cdot Jin_{\psi_{5}}(\langle (1,1) \rangle \times C_5) $,   & $
2$ & $
4$ & $
2$ & $
2$ \\
    $C_{2} \cdot Jin_{\psi_{6}}(\langle (1,1) \rangle \times C_5) $,   & $
2$ & $
4$ & $
2$ & $
2$ \\
    $C_{4} \cdot Jin_{\psi_{7}}(C_2 \times C_2 \times \{0\} ) $,   & $
1$ & $
10$ & $
4$ & $
1$ \\
    $C_{8} \cdot (\langle (1,1) \rangle \times \{0\} ) $,   & $
1$ & $
10$ & $
4$ & $
1$ \\
    $C_{4} \cdot Jin_{\psi_{9}}(C_2 \times C_2 \times \{0\} ) $,   & $
1$ & $
10$ & $
4$ & $
1$ \\
    $C_{8} \cdot(C_2 \times \{0\}  \times \{0\} ) $,   & $
1$ & $
10$ & $
4$ & $
1$ \\
    $C_{4} \cdot Jin_{\psi_{11}}(C_2 \times C_2 \times \{0\} ) $,   & $
1$ & $
10$ & $
4$ & $
1$ \\
    $C_{8} \cdot(\{0\} \times C_2 \times \{0\} ) $,   & $
1$ & $
10$ & $
4$ & $
1$ \\
    $C_{4} \cdot Jin_{\psi_{13}}(C_2 \times C_2 \times \{0\} ) $,   & $
1$ & $
10$ & $
4$ & $
1$ \\
    $C_{8} \cdot(C_2 \times C_2 \times \{0\} ) $,   & $
1$ & $
5$ & $
4$ & $
1$ \\
    $C_{4} \cdot Jin_{\psi_{15}}(C_2 \times C_2 \times C_5) $,   & $
1$ & $
2$ & $
1$ & $
1$ \\
    $C_{8} \cdot(\langle (1,1) \rangle \times C_5) $,   & $
1$ & $
2$ & $
1$ & $
1$ \\
    $C_{4} \cdot Jin_{\psi_{17}}(C_2 \times C_2 \times C_5) $,   & $
1$ & $
2$ & $
1$ & $
1$ \\
    $C_{8} \cdot(C_2 \times \{0\}  \times C_5) $,   & $
1$ & $
2$ & $
1$ & $
1$ \\
    $C_{4} \cdot Jin_{\psi_{19}}(C_2 \times C_2 \times C_5) $,   & $
1$ & $
2$ & $
1$ & $
1$ \\
    $C_{8} \cdot (\{0\} \times C_2 \times C_5) $,   & $
1$ & $
2$ & $
1$ & $
1$ \\
    $C_{4} \cdot Jin_{\psi_{21}}(C_2 \times C_2 \times C_5) $,   & $
1$ & $
2$ & $
1$ & $
1$ \\
    $C_{8} \cdot (C_2 \times C_2 \times C_5) $,   & $
1$ & $
1$ & $
1$ & $
1$ \\
\bottomrule
    \end{tabular}
\end{table}

\begin{equation}
    \mathbb{F}_3 \hat{G} \cong \mathbb{F}_{3}^{8} \oplus \mathbb{F}_{3^{4}}^{8} \oplus \text{Mat}_{2}(\mathbb{F}_{3})^{2} \oplus \text{Mat}_{2}(\mathbb{F}_{3^{2}})^{2} \oplus \text{Mat}_{2}(\mathbb{F}_{3^{4}})^{6}
\end{equation}

Now, thanks to the idempotents of $\mathbb{F}_3 \hat{G}$ we can check that

\begin{align*}
    I&= \langle (\lambda^2,(0,0,0)) + 2(\lambda,(0,0,0))+ 2(1,(0,0,0))\rangle  \\&\simeq \mathbb{F}_{3}^{8} \oplus \mathbb{F}_{3^{4}}^{8} \oplus \text{Mat}_{2}(\mathbb{F}_{3})^{2} \oplus \text{Mat}_{2}(\mathbb{F}_{3^{2}}) \oplus \text{Mat}_{2}(\mathbb{F}_{3^{4}})^{4}
\end{align*}

And therefore 

\begin{equation}
   \mathbb{F}_9^{\sigma}[G,\alpha]\cong\mathbb{F}_3 \hat{G}/ I \cong \text{Mat}_{2}(\mathbb{F}_{3^{2}}) \oplus \text{Mat}_{2}(\mathbb{F}_{3^{4}})^{2}
\end{equation}

As we wanted to compute. Note that the dimension over $\mathbb{F}_3$ match, as 
$\dim_{\mathbb{F}_3}(\mathbb{F}_9^{\sigma}[G,\alpha]) = 2\cdot |G| = 40$ and $\dim_{\mathbb{F}_3}(\text{Mat}_{2}(\mathbb{F}_{3^{2}}) \oplus \text{Mat}_{2}(\mathbb{F}_{3^{4}})^{2})=40$

\section*{Acknowledgements}
This research was supported by the Spanish Ministry of Science, Innovation, and Universities under the FPU 2023 grant programme, and the project PID2023-149203NB-I00.

\bibliographystyle{cas-model2-names}

\bibliography{cas-refs}



\end{document}